\documentclass[reqno]{amsart}
\usepackage{amssymb,amsmath,hyperref}
\usepackage{amsrefs}
\usepackage[foot]{amsaddr}
\usepackage{bbold,stackrel}
 
\newtheorem{thm}{Theorem}

\newtheorem{cor}[thm]{Corollary}
\newtheorem{defi}[thm]{Definition}
\newtheorem{rem}[thm]{Remark}
\newtheorem{nota}[thm]{Notation}

\newtheorem{princ}[thm]{Principle}

\newtheorem{ack}[thm]{Acknowledgement}

\newtheorem*{tempo*}{Template}

\newcommand\be{\begin{equation}}
\newcommand\ee{\end{equation}} 

\usepackage{amsmath,amsfonts} 

\usepackage[applemac]{inputenc}

\def\bdefi{\begin{defi}\rm}
\def\edefi{\end{defi}}
\def\bnota{\begin{nota}\rm}
\def\enota{\end{nota}}
\def\FIVE{\Pi_{1}^{1}\text{-\textup{\textsf{CA}}}_{0}}

\def\SIX{\Pi_{2}^{1}\text{-\textsf{\textup{CA}}}_{0}}

\def\SIXK{\Pi_{k}^{1}\text{-\textsf{\textup{CA}}}_{0}^{\omega}}

\def\ATR{\textup{\textsf{ATR}}}

\def\Z{\textup{\textsf{Z}}}

\def\ZFC{\textup{\textsf{ZFC}}}

\def\TIET{\textup{\textsf{TIE}}}
\def\TIE{\textup{\textsf{TIE}}}

 \def\r{\mathbb{r}}

\def\u{\textup{\textsf{u}}}

\def\RCA{\textup{\textsf{RCA}}}
\def\({\textup{(}}
\def\){\textup{)}}

\def\RCAo{\textup{\textsf{RCA}}_{0}^{\omega}}
\def\ACAo{\textup{\textsf{ACA}}_{0}^{\omega}}

\def\WKL{\textup{\textsf{WKL}}}

\def\WWKL{\textup{\textsf{WWKL}}}
\def\bye{\end{document}}
\def\N{{\mathbb  N}}
\def\Q{{\mathbb  Q}}
\def\R{{\mathbb  R}}
\def\L{\textsf{\textup{L}}}

\def\MUC{\textup{\textsf{MUC}}}

\def\di{\rightarrow}
\def\asa{\leftrightarrow}
\def\ACA{\textup{\textsf{ACA}}}

\def\QFAC{\textup{\textsf{QF-AC}}}

\def\HBU{\textup{\textsf{HBU}}}
\def\PRA{\textup{\textsf{PRA}}}
\def\cocode{\textup{\textsf{cocode}}}
\def\CT{\textup{\textsf{CT}}}
\def\BW{\textup{\textsf{BW}}}
\def\RCP{\textup{\textsf{RCP}}}
\def\HEI{\textup{\textsf{HEI}}}
\def\WEI{\textup{\textsf{WEI}}}
\def\tot{\textup{\textsf{tot}}}
\def\cont{\textup{\textsf{cont}}}
\def\FVP{\textup{\textsf{FVP}}}

\def\SS{\textup{\textsf{S}}}

\def\WEI{\textup{\textsf{WEI}}}

\def\LIN{\textup{\textsf{LIN}}}
\def\WHBU{\textup{\textsf{WHBU}}}

\def\ORD{\textup{\textsf{ORD}}}

\def\UATR{\textup{\textsf{UATR}}}

\def\FF{\textup{\textsf{FF}}}

\def\ECF{\textup{\textsf{ECF}}}

\newcommand{\T}{\mathcal{T}}

\usepackage{graphicx}
\usepackage{tikz}
\usetikzlibrary{matrix, shapes.misc}

\numberwithin{equation}{section}
\numberwithin{thm}{section}

\usepackage{comment}

\begin{document}
\title{Representations and the Foundations of Mathematics}
\author{Sam Sanders}
\address{Department of Philosophy II, RUB Bochum, Germany}
\email{sasander@me.com}
\begin{abstract}
The representation of mathematical objects in terms of (more) basic ones is part and parcel of (the foundations of) mathematics.  
In the usual foundations of mathematics, i.e.\ $\ZFC$ set theory, all mathematical objects are represented by sets, while ordinary, i.e.\ non-set theoretic, mathematics 
is represented in the more parsimonious language of second-order arithmetic.  This paper deals with the latter representation for the rather basic 
case of continuous functions on the reals and Baire space.  We show that the logical strength of basic theorems named after Tietze, Heine, and Weierstrass, changes significantly upon the 
replacement of `second-order representations' by `third-order functions'.  
We discuss the implications and connections to the \emph{Reverse Mathematics} program and its foundational claims regarding predicativist mathematics and Hilbert's program for the foundations of mathematics. 
Finally, we identify the problem caused by representations of continuous functions and formulate a criterion to avoid problematic codings within the bigger picture of representations. 
\end{abstract}


\maketitle
\thispagestyle{empty}
\section{Introduction}\label{intro}
Lest we be misunderstood, let our first order of business be to formulate the following blanket caveat:  
\begin{center}
\emph{any formalisation of mathematics generally involves some kind of representation \(aka coding\) of mathematical objects in terms of others}.  
\end{center}
Now,  the goal of this paper is to critically examine the role of representations based on the language of second-order arithmetic; such an examination perhaps unsurprisingly involves the comparison of theorems based on second-order representations versus theorems formulated in third-order arithmetic.  
To be absolutely clear, we do not claim that the latter represent the ultimate mathematical truth, nor do we (wish to) downplay the role of representations in third-order arithmetic.  

\smallskip

As to content, we briefly introduce representations based on second-order arithmetic in Section~\ref{stateofart}, while our 
main results are sketched in Section \ref{maintro}.

\subsection{Representations and second-order arithmetic}\label{stateofart}
The representation of mathematical objects in terms of (more) basic ones is part and parcel of (the foundations of) mathematics.  
For instance, in the usual foundation of mathematics, $\ZFC$ set theory, the natural numbers $\N=\{0,1, 2, \dots\}$ can be represented via sets
by the \emph{von Neumann ordinals} as follows: $0$ corresponds to the empty set $\emptyset$, and $n+1$ corresponds to $n\cup \{n\}$.
The previous is more than just a simple technical device: critique of this sort of representation has given rise to various structuralist and nominalist foundational philosophies \cite{bijnaserafijn}.  

\smallskip

A more conceptual example is the vastly more parsimonious (compared to set theory) language $\L_{2}$ of \emph{second-order arithmetic}, essentially consisting of natural numbers and sets thereof.  
Nonetheless, this language is claimed to allow for the formalisation of \emph{ordinary}\footnote{Simpson describes \emph{ordinary mathematics} in \cite{simpson2}*{I.1} as \emph{that body of mathematics that is prior to or independent of the introduction of abstract set theoretic concepts}.} or core mathematics by e.g.\ Simpson as follows.  
\begin{quote}
The language $\L_{2}$ comes to mind because it is just adequate to define the majority of ordinary
mathematical concepts and to express the bulk of ordinary mathematical
reasoning. \cite{simpson2}*{I.12}
\end{quote}
\begin{quote}
We focus on the language of second-order arithmetic, because that language is the
weakest one that is rich enough to express and develop the bulk of core mathematics. \cite{simpson2}*{Preface}
\end{quote}
Simspon's claims can be substantiated via a number of representations of uncountable objects as second-order objects, including (continuous) functions on $\R$ \cite{simpson2}*{II.6}, topologies \cite{mummy}, metric spaces \cite{simpson2}*{I.8.2}, measures \cite{simpson2}*{X.1},  et cetera.  One usually refers to these second-order representations as `codes' and the associated practise of using codes as `coding'.  The aim of this paper is a critical investigation of this coding practise.    
 
\smallskip

To be absolutely clear, we do not judge the correctness of Simpson's above grand claims in this paper, nor will we make any general judgement as to whether coding is good or bad \emph{per se}.   
What we shall show is much more specific, namely that introducing codes \emph{for continuous functions} in certain theorems of core or ordinary mathematics \emph{changes the minimal axioms required to prove these theorems}.  The knowledgable reader of course recognises the final italicised sentence as follows.
\begin{quote} 
We therefore formulate our \emph{Main Question} as follows: \emph{Which set existence axioms are needed to prove the theorems of ordinary, non-set-theoretic mathematics?} 
\cite{simpson2}*{I.1}; emphasis original.
\end{quote}  
Indeed, the previous question is the central motivation of the \emph{Reverse Mathematics} program (RM hereafter), founded by Friedman in \cite{fried, fried2}.
An overview of RM may be found in \cites{simpson1,simpson2}, while an introduction for the `scientist in the street' is \cite{stillebron}.  
We provide a brief overview of Reverse Mathematics, including Kohlenbach's higher-order variety, in Section \ref{janarn}. 

\smallskip

Thus, if one is interested in the Main Question of RM, it seems \emph{imperative} that the aforementioned coding practise does not change the minimal axioms needed to prove a given theorem. 
This becomes even more pertinent in light of Simpson's strong statements regarding the use of `extra data' in constructive mathematics.
\begin{quote}
The typical constructivist response to a nonconstructive mathemati-
cal theorem is to modify the theorem by adding hypotheses or “extra
data”.  In contrast, our approach in this book is to analyze the provability of mathematical theorems as they stand, passing to stronger
subsystems of $\Z_{2}$ if necessary. \cite{simpson2}*{I.8}
\end{quote}
\begin{quote}
This situation has prompted some authors, for example
Bishop/Bridges [20, page 38], to build a modulus of uniform continuity into their definitions of continuous function.
Such a procedure may be appropriate for Bishop since his goal is to replace ordinary mathematical theorems by
their ``constructive” counterparts.
However, as explained in chapter I, our goal is quite different [from Bishop's Constructive Analysis].
Namely, we seek to draw out the set existence assumptions which are implicit in the ordinary mathematical theorems
\emph{as they stand}. \cite{simpson2}*{IV.2.8}; emphasis original.
\end{quote}
Hence, if one takes seriously the claim that RM studies theorems of mathematics `as they stand', it is of the utmost important that the latter standing is not changed by the coding practise of RM.  
Indeed, Kohlenbach has shown in \cite{kohlenbach4}*{\S4} that the existence of a code for a continuous function is the same as the latter having a modulus of continuity, a slight constructive enrichment going against the above.

\smallskip
   
Moreover, coding is not a tangential topic in RM in light of the following:
the two pioneers of RM, Friedman and Simpson, actually devote a section titled `The Coding Issue' to the coding practise in RM in their  `issues and problems in RM' paper \cite{fried5}, from which we consider the following.
\begin{quote}
Most mathematics naturally lies within the realm of complete separable metric
spaces and continuous functions between them defined on open, closed, compact or
$G_{\delta}$ subsets. This is the central coding issue.
\begin{quote}
PROBLEM. Continuation of the previous problem: Show that 
Simpson's neighborhood condition coding of partial continuous functions
between complete separable metric spaces is "optimal". (It amounts
to a coding of continuous functions on a $G_{\delta}$.) 
\end{quote}
We emphasize our view that the handling of the critical coding in $\RCA_{0}$ in [Simpson's monograph \cite{simpson2}]
is canonical, but we are asking for theorems supporting this view.  \cite{fried5}*{p.\ 135}
\end{quote}
Moreover, in his `open problems in RM' paper, Montalb\'an discusses Friedman's \emph{strict RM} program \cite{friedlc06} as follows. 
\begin{quote}
Friedman has proposed the development what he calls strict reverse mathematics (SRM). The objective of SRM is to eliminate the following two possible criticisms of reverse mathematics: that we need to code objects in cumbersome ways (something that is not part of classical mathematics), [\dots]. 

\smallskip

As for the coding issue, Friedman says that, for each area X of mathematics, there will be a SRM for X. The basic concepts of X will be taken as primitives, avoiding the need for coding. This would also allow consideration of uncountable structures, thereby getting around this limitation of reverse mathematics. \cite{montahue}*{p.\ 449}
\end{quote}
The previous quotes suggest that on one hand the coding practise of RM is definitely on the radar in RM, while on the other hand there is an optimistic belief that coding does not cause any major problems in RM.  
This brings us to the next section in which we discuss the main results of this paper.  

\subsection{Main results}\label{maintro}
The conclusion from the previous section was that on one hand the coding practise of RM is definitely on the radar in RM, while on the other hand there is an optimistic belief that coding does not cause any major problems in RM.  
Our main goal is to dispel this belief in that we identify basic theorems about continuous functions, including basic ones named after Heine, Tietze, and Weierstrass, where the (only) change from `continuous third-order function' to `second-order code for continuous function' changes the strength of the theorem at hand, namely as in Figure \ref{kk} below.  
Since we want to consider objects from third-order arithmetic, we shall work in Kohlenbach's higher-order RM introduced in \cite{kohlenbach2} and discussed in Section \ref{prelim1}.  
The associated base theory is $\RCAo$, which proves the same $\L_{2}$-sentence as $\RCA_{0}$, up to insignificant logical details, by Remark \ref{ECF}. 

\smallskip

Figure \ref{kk} below provides an overview of our results to be proved in Section \ref{tietenzien}-\ref{ekesect}. The first column lists a theorem, while the second (resp.\ third) column lists the theorem's logical strength for the formulation involving second-order codes (resp.\ third order functions).     
The reader should (only) view these two formulations as two possible formalisations of the same theorem, based on respectively second-order and third-order arithmetic.   
In particular, we discuss the question whether third-order objects are a more natural form of representation than second-order objects in Section \ref{REF}.
We wish to stress that it would perhaps not come as a surprise that the coding of topologies or measurable sets in second-order arithmetic has its problems \cites{hunterphd, dagsamVI}, but the theorems in Figure \ref{kk} are \emph{quite} elementary in comparison. 

\smallskip

 Note that we use the usual RM-definition of (separably) closed sets from \cite{simpson2}*{II.4} and \cite{browner, brownphd} everywhere, i.e.\ the only difference between the second and third column in Figure \ref{kk} is the change from `(continuous) third-order function' to `second-order code'.  
 We also note that the final rows are formulated over Cantor and Baire space, i.e.\ the coding of the reals is not relevant here, and we could obtain versions of the other theorems for the latter spaces. 
 Finally, the first instance of Ekeland's variational principle involves \emph{honest}\footnote{Honest codes are a (rather) technical device from \cite{ekelhaft}*{Def.\ 5.1}. \label{core}} codes.  
 
 \begin{figure}[h]
\begin{tabular}{c|c|c}
 & RM-codes ($\L_{2}$)  & Third-order functions  \\
\hline
  Heine's theorem for  &equivalent  & provable    \\
  separably closed sets & to $\ACA_{0}$ & in $\RCAo+\WKL_{0}$\\
  \hline
    Weierstrass' theorem for  &equivalent  & provable    \\
  separably closed sets & to $\ACA_{0}$ & in $\RCAo+\WKL_{0}$\\
  \hline
 Tietze's theorem for &  equivalent & provable  \\
 separably closed sets & to $\ACA_{0}$ & in $\RCAo$\\
  \hline
One-point-extension &  equivalent & provable  \\
 theorem (Theorems \ref{KAYO} and \ref{KAYO2}) & to $\ACA_{0}$ & in $\RCAo$\\
 \hline
 Ekeland's variational & equivalent & provable in a \\
 principle$^{\ref{core}}$ on $2^{\N}$  & to $\ACA_{0}$  & conservative extension \\
 && of $\RCAo+\WKL$\\
 \hline
 Ekeland's variational & equivalent & provable in a \\
 principle on $\N^{\N}$  & to $\FIVE$ & conservative extension \\
 && of $\ACA_{0}$ 
\end{tabular}
\caption{Summary of our results}
\label{kk}
\end{figure}\vspace{-0.4cm}~\\
We caution the reader not to over-interpret the above: for instance, we do not claim that the third column of Figure \ref{kk} provides the ultimate (RM) analysis of the theorems in the first column.
What Figure \ref{kk} does establish is that coding can change the logical strength of the most basic theorems.  Hence, coding as done in RM is neither good nor bad \emph{in general}, but \emph{in the specific case} of the Main Question of RM, coding seems problematic \emph{lest the wrong minimal axioms be identified}.  

\smallskip

Indeed, the foundational claims made in relation to RM are (of course) based on the correct identification of the minimal axioms need to prove a given theorem.  
We are thinking of the development of Russell-Weyl-Feferman \emph{predicative mathematics} or Simpson's partial realisation of \emph{Hilbert's program for the foundations of mathematics}. 
In Section \ref{FRM}, we introduce the aforementioned programs and critically investigate the implications of Figure \ref{kk} for these claims and programs 

\smallskip

Now, the above is only a stepping stone: we wish to understand \emph{why} the theorems in Figure \ref{kk} behave as they do.  
Moreover, this understanding should lead to a distinction between `good' and `bad' codes, as e.g.\ the coding of real numbers does not seem to lead to results as in Figure \ref{kk}.  
We discuss these matters in some detail in Section \ref{XMX}.  
In a nutshell, the cause of the results in Figure \ref{kk} is that there are partial codes that do not correspond to a third-order object in a weak system.  
Digging more deeply, we shall observe that second- and third-order objects are intimately connected, also in non-classical extensions of the base theory, but the same is not true for (second-order) \emph{representations} of third-order objects.  
The absence of this connection is what makes a code `bad' in the sense of yielding results as in Figure~\ref{kk}.

\smallskip

Finally, to establish the results in Figure \ref{kk} we require the `excluded middle trick' introduced in \cite{dagsamV}.  
Since all systems are based on classical logic, one can always invoke the following disjunction.
\begin{quote}
Either there exists a discontinuous function on $\R$ (in which case we have higher-order $\ACA_{0}$) or all function on $\R$ are continuous.  
\end{quote}
This is the crucial step for the below proofs, as will become clear. 

\smallskip

In conclusion, the aim of this paper is to establish the results in Figure \ref{kk} and discuss the associated foundational implications.  
At the risk of repeating ourselves, Figure \ref{kk} is meant to showcase the difference between the second-order and third-order formalisations of the theorems in the first column. 
To be absolutely clear, the third-order formalisation can also be said to involve a kind of representation/coding, i.e.\ the former does not constitute the theorem \emph{an sich}. We should not have to point out that the latter point of view is captured by the centred statement in Section \ref{intro}.

\section{Tietze extension theorem}\label{tietenzien}
We establish the results regarding the Tietze extension theorem as in Figure \ref{kk}.
Hereafter, we assume familiarity with the basic notions of RM, Kohlenbach's base theory $\RCAo$ in particular.  
For the reader's convenience, an introduction to RM, and the definition of $\RCAo$, can be found in Section \ref{janarn}.
We like to point out that our results do not call into question the \emph{Big Five} phenomenon, as will become clear. 
\subsection{Introduction}\label{trintro}
We discuss some relevant facts for Tietze's theorem.  

\smallskip

First of all, the Tietze extension theorem has been studied in RM in e.g.\ \cite{simpson2, brownphd, withgusto, paultiet}.  
In the latter references, it is shown that there are versions of Tietze's extension theorem provable in $\RCA_{0}$, equivalent to $\WKL_{0}$, and equivalent to $\ACA_{0}$.  
These different versions make use of different notions of continuity and closed set.  
The main result of this section is that a slight change to the Tietze's extension theorem, namely replacing second-order codes by third-order functions, makes the version at the level of $\ACA_{0}$ provable in $\RCAo$.  

\smallskip

Secondly, the aforementioned `slight change' definitely has historical antecedent, in that it can be found in Tietze's paper \cite{tietze}, and is actually the dominating formalism therein.   
To see this, let us discuss the exact formulation of Tietze's extension theorem, which expresses that \emph{for certain spaces $X$, if a function $f$ is continuous on a closed $C\subset X$, there is a function $g$, continuous on all of $X$, such that $f=g$ on $C$}.
%
There are at least two ways of formulating the antecedent of the previous theorem, say for $\R\di \R$-functions, namely as follows.  
\begin{enumerate}
\item[(A)] The function $f$ is defined and continuous on $C$; undefined on $\R\setminus C$. 
\item[(B)] The function $f$ is defined everywhere on $\R$ and continuous on $C$.  
\end{enumerate}
Tietze proves three theorems in \cite{tietze}, called \emph{Satz 1, 2, \textup{and} 3}.  The first two are formulated using (B), while the third one is formulated as a corollary to the first and second theorem, and is formulated using (A).
Tietze explicity mentions that $f$ can be discontinuous outside of $C$ in \cite{tietze}*{p.\ 10}.
When treating Tietze's extension theorem, Carath\'eodory uses both (A) and (B) in \cite{carapils}*{\S\S541-543}, while Hausdorff states that (B) is used in \cite{hausen}.

\smallskip

It goes without saying that item (A) corresponds rather well to the RM-definition of `continuous function on a closed set', while (B) is essentially the higher-order definition since all objects are total (= everywhere defined) in Kohlenbach's RM.
Moreover, we show below that the `actual' strength of the second-order Tietze extension theorem comes from the fact that it can extend certain (partial) codes into (total) higher-order functions (working in $\RCAo$).

\smallskip

In conclusion, given the previous observation regarding items (A) and (B), it seems reasonable to study a version of Tietze's extension theorem based on (B) in higher-order RM, which can be found in Section \ref{sepclo}.  
Hereonafter, `continuous' refers to the usual `$\epsilon$-$\delta$' definition of third-order functions, while `RM-continuous' refers to the coding used in RM \cite{simpson2}*{II.6.1}.


\subsection{Separably closed sets}\label{sepclo}
We establish the results sketched in Figure \ref{kk}, i.e.\ we study Tietze's extension theorem for \emph{separably closed} sets.  

\smallskip

First of all, the aforementioned theorem can be be found in \cite{brownphd, withgusto} and we note that `separably closed sets' are closed sets represented by a countable dense sub-set (see e.g.\ \cite{brownphd,browner} for details).   
\bdefi[$\RCA_{0}$; separably closed set]\label{kink}
A sequence $S=(x_{n})_{n\in \N}$ is a (code for a) \emph{separably closed set} $\overline{S}$ in a complete separable metric space $\hat{A}$.  
We say that $x\in \hat{A}$ \emph{belongs to} $\overline{S}$, denoted `$x\in \overline{S}$' if $(\forall k\in \N)(\exists n\in \N)(d(x, x_{n})<\frac{1}{2^{k}} )$. 
\edefi
The second-order version of Tietze's theorem is then as follows. 
\begin{princ}[$\TIET^{2}$]
For $f:C\di \R$ {RM-continuous} on the {separably closed} $C\subseteq [0,1]$, there is {RM-continuous} $g:[0,1]\di \R$ such that $f(x)=_{\R}g(x)$ for $x\in C$.
\end{princ}
As suggested above, $\TIET^{2}\asa \ACA_{0}$ over $\RCA_{0}$ by \cite{withgusto}*{Theorem~6.9}.  
The exact choice of domain does not seem to matter much by \cite{brownphd}*{Theorem~1.35}.

\smallskip
 
Secondly, the higher-order version version of Tietze's theorem based on item (B) from Section \ref{trintro} is as follows. 
The reader will agree that the only change from $\TIET^{2}$ to $\TIET^{\omega}$ is that we replaced `second-order code' by `third-order function'.
\begin{princ}[$\TIET^{\omega}$]
For $f:[0,1]\di \R$ {continuous} on the {separably closed} $C\subseteq[0,1]$, there is {continuous} $g:[0,1]\di \R$ such that $f(x)=_{\R}g(x)$ for $x\in C$.
\end{princ}
To be absolutely clear, recall we use `continuous' in the sense of the usual `$\epsilon$-$\delta$' definition, while `RM-continuous' refers to the coding used in RM \cite{simpson2}*{II.6.1}.
In contrast to $ \TIET^{2}$ requiring arithmetical comprehension, the higher-order version $\TIET^{\omega}$ is much weaker. 
\begin{thm}\label{WLS}
The system $\RCAo$ proves $\TIET^{\omega}$.
\end{thm}
\begin{proof}
We use the law of excluded middle as in $(\exists^{2})\vee \neg(\exists^{2})$.  
In case $\neg(\exists^{2})$, all functions on $\R$ are continuous by \cite{kohlenbach2}*{\S3}, and we may use $g=f$ to obtain $\TIET^{\omega}$. 
In case $(\exists^{2})$, we use \cite{brownphd}*{Theorem 1.10} to guarantee that any seperably closed set $C\subseteq \R$ is also closed in the sense of RM, i.e.\ given by a $\Pi_{1}^{0}$-formula in $\L_{2}$.
Note that $\exists^{2}$ can decide such formulas.  
By \cite{kohlenbach4}*{\S4}, we can obtain an RM-code for $f$ on $C$ using $\exists^{2}$.   Applying $\TIET^{2}$ to the RM-code of $f$, there is \emph{code for} a continuous function $g$ such that $f=_{\R}g$ on $C$.   
Apply $\QFAC^{1,0}$ to the totality of $g$ on $[0,1]$ to obtain the required continuous function.
\end{proof} 
\noindent
The above results can be neatly summarised as in \eqref{kut}, working over $\RCAo$:
\be\label{kut}
\TIET^{2}\di \ACA_{0} \di \WKL_{0}\di \TIET^{\omega}
\ee
%
Thirdly, we study a `reverse coding principle' that allows one to obtain a (higher-order) function from certain codes.  
\begin{princ}[$\RCP$]
For a RM-code $f$ defined on a separably closed set $C\subseteq [0,1]$, there is $h:[0,1]\di \R$ with $(\forall x\in C)(h(x)=_{\R}f(x))$.
\end{princ}
Note that a \emph{total} RM-code trivially\footnote{Just apply $\QFAC^{1,0}$ to the formula `$\alpha(x)$ is defined for all $x\in \R$' for the RM-code $\alpha$.}  yields a higher-order function with the same values, in contrast to the following theorem.
\begin{thm}\label{winkel}
The system $\RCAo$ proves $\RCP\di \ACA_{0}$.
\end{thm}
\begin{proof}
Let $f$ be a RM-code defined on a separably closed $C\subset [0,1]$.
Then $\RCP$ yields $h:[0,1]\di \R$ which is continuous on $C$.
Applying $\TIET^{\omega}$ to $h$ (see Theorem~\ref{WLS}) then yields a (higher-order) version of $\TIET^{2}$ for $f$ given by RM-codes, but where $g:[0,1]\di \R$ is not given by a RM-code.  
However, the proof of $\TIET^{2}\di \ACA_{0}$ in \cite{withgusto}*{Theorem 6.9} still goes through when $g:[0,1]\di\R$ is not given by a code but is only continuous on $[0,1]$.  
This proof uses $\Pi_{1}^{0}$-separation and we note that $\RCAo$ proves separation for $\Pi_{1}^{0}$-formulas \emph{with type two parameters} in the same way as $\RCA_{0}$ proves $\Pi_{1}^{0}$-separation for $\L_{2}$-formulas.  
The latter result is in \cite{simpson2}*{IV.4.8}.  
\end{proof}
The previous proof becomes much easier if we work in $\RCAo+\WKL$, as it is established in \cite{kohlenbach4}*{\S4} that the latter system shows 
that a continuous $Y^{2}$ has a code on Cantor space; the same goes through for $[0,1]$ and hence $g$ from the proof.
We shall repeatedly use this `coding principle' in the remainder of this paper. 
We now consider the following somewhat strange corollary.
\begin{cor}\label{fledna}
The following are equivalent over $\RCAo$ to $\ACA_{0}$:
\begin{enumerate}
 \renewcommand{\theenumi}{\alph{enumi}}
\item $\TIET^{2}$
\item $\RCP$\label{frak}
\item {For a RM-code $f$ defined on a separably closed set $C\subseteq [0,1]$, there is \emph{continuous} $g:[0,1]\di \R$ such that $(\forall x\in C)(g(x)=_{\R}f(x))$}.\label{taal}
\end{enumerate}
\end{cor}
\begin{proof}
Since $\TIET^{2}\asa \ACA_{0}$ by \cite{withgusto}*{Theorem 6.9}, we only need to prove $\eqref{frak}\di \eqref{taal}$ in light of the theorem.  In case $\neg(\exists^{2})$, all functions on $\R\di \R$ are continuous by \cite{kohlenbach2}*{\S3}.  
In case $(\exists^{2})$, we also have $\ACA_{0}$ and applying $\TIET^{2}$ yields \emph{a code for} a continuous $g:\R\di \R$ as in item \eqref{taal}.  
This code in turn yields the function required for item \eqref{taal} by applying $\QFAC^{1,0}$ to the totality of this code.  
The law of excluded middle now finishes the proof. 
\end{proof}
We note that the results in Corollary \ref{fledna} are extremely robust: we can endow $g$ in $\RCP$ with \emph{any} property weaker than continuity and the equivalence would still go through. 
Moreover, we observe that $\RCP$ is the weakest set existence axiom equivalent to $\TIET^{2}$ over $\RCAo$, while 
the former constitutes the non-constructive part of the Tietze extension theorem as in $\TIET^{2}$.  
Indeed, working over $\RCAo$, once the original function as in $\TIET^{2}$ is made total via $\RCP$, there is no additional
strength required to obtain a continuous extension function, as $\TIET^{\omega}$ is provable in $\RCAo$.
%

\smallskip

Finally, we show that item (A) from Section \ref{trintro} is rather strong when the associated notions are interpreted in second-order RM.  
For simplicity, we work over Cantor space $2^{\N}$.
\begin{thm}[$\RCAo$]
Let $D\subset 2^{\N}$ be a RM-closed set and let $f:D\di \N$ be a RM-code for a continuous function.
Then $(\forall x\in 2^{\N})(x\in D\asa \textup{$f(x)$ is defined})$ implies that `$x\in D$' is decidable, i.e.\ there is $\Phi^{2}$ such that $(\forall x\in 2^{\N})(x\in D\asa \Phi(x)=0)$.
\end{thm}
\begin{proof}
Note that `$x\in D$' is $\Pi_{1}^{0}$ while `$f(x)$ is defined' is $\Sigma_{1}^{0}$.  Applying $\QFAC^{1,0}$ to the forward implication yields the required $\Phi^{2}$.  
\end{proof}
We briefly discuss the foundational implications of the above, while a more detailed investigation may be found in Section \ref{FRM}.
Simpson states in \cite{simpson2}*{IX.3.18} that the partial realisation of Hilbert's program for the foundations of mathematics is a `very important direction of research'.
In a nutshell, the equivalence $\TIE^{2}\asa \ACA_{0}$ suggests that the Tietze extension theorem for separably closed sets is out of reach for this partial realisation, while 
Theorem~\ref{WLS} implies \emph{that this is not the case}.  

\smallskip

Finally, we study a  `one-point-extension theorem' from \cite{yokoyama3}*{\S3}, which was suggested to us by K.\ Yokoyama.
The reversal in Theorem \ref{KAYO} is attributed to K.\ Tanaka in \cite{yokoyama3}.    
\begin{thm}\label{KAYO}
The following assertions are pairwise equivalent over $\RCA_{0}$. 
\begin{enumerate}
\item  $\ACA_{0}$.
\item If $f$ is a RM-continuous function from $(0, 1)$ to $\R$ such that $\lim_{x\di +0}f(x)=0$, 
then there exists a RM-continuous function $\tilde{f}$ from $[0, 1)$ to $\R$ such that $\tilde{f}(x)=f(x)$ for $x\in (0,1)$ and $\tilde{f}(0)=0$.
\end{enumerate}
\end{thm}
This theorem is interesting, as the `extension' of $f$ consist only of one (obvious) point, while separably closed sets are not used. 
In contrast to $\ACA_{0}$ in the previous theorem, we have the following theorem.
\begin{thm}[$\RCAo$]\label{KAYO2}
If $f:\R\di \R$ is continuous on $(0, 1)$ and such that $\lim_{x\di +0}f(x)=0$, 
then there exists a function $\tilde{f}:\R\di \R$ such that $\tilde{f}(x)=f(x)$ for $x\in (0,1)$ and $\tilde{f}(0)=0$, i.e.\ $\tilde{f}$ is continuous on $[0, 1)$.
\end{thm}
\begin{proof}
We make use of $(\exists^{2})\vee \neg(\exists^{2})$.  In the latter case, all functions on $\R$ are continuous by \cite{kohlenbach2}*{\S3} and $f=\tilde{f}$.  
In the former case, define $\tilde{f}(x)$ as $0$ if $x=0$ and $f(x)$ otherwise, using $\exists^{2}$.  Clearly, $\tilde{f}$ is continuous on $[0,1)$.  
\end{proof}
In conclusion, we have established the results from Figure \ref{kk} pertaining to the Tietze extension theorem.  
In doing so, we have identified $\RCP$ as the source of `non-constructivity' in $\TIET^{2}$.   Our results were proved using the `excluded middle trick' described in Section \ref{maintro}.  
This kind of proof suggests that third-order functions $f$ as in $\TIET^{\omega}$ come with a certain kind of constructive enrichment: the extension $g$ is implicit in $f$ via a simple (classical) case distinction.  
In particular, the weakness of $\TIET^{\omega}$ relative to $\TIET^{2}$ is due to this  additional (implicit) information.  
The reader may interpret this as $\TIET^{2}$ and $\TIET^{\omega}$ being variations of the same theorem with different coding or constructive enrichment.
We discuss the latter view in more detail in Section \ref{REF}, as well as the question whether second-order or third-order representations are more natural. 

\section{Theorems by Heine and Weierstrass}
We establish the results sketched in Figure \ref{kk} for Heine's \emph{continuity theorem} and Weierstrass' \emph{approximation theorem}.  
These theorems are studied in RM in \cite{simpson2}*{IV.2}.
We shall study these theorems for separably closed sets; our results are similar to those for Tietze's theorem and we shall therefore be brief.

\smallskip

First of all, the following two principles are the $\L_{2}$-versions of the Heine and Weierstrass theorems for separably closed sets.
\begin{princ}[$\HEI^{2}$]
 Let $C\subseteq[0,1]$ be separably closed.  Any RM-continuous $f:C \di \R$ is also uniformly continuous with a modulus on $C$.
 \end{princ}
\begin{princ}[$\WEI^{2}$]
 Let $C\subseteq[0,1]$ be separably closed.  For RM-continuous $f:C\di \R$, there is a sequence of polynomials $(p_{n})_{n\in \N}$ with $(\forall x\in C)(|f(x)-p_{n}(x)|<\frac{1}{2^{n}})$.
 \end{princ}
We have the following expected second-order RM result.
 \begin{thm}\label{klanty}
  The system $\RCA_{0}$ proves $\ACA_{0}\asa \HEI^{2}\asa \WEI^{2}$.
\end{thm}  
\begin{proof}
We have $\TIET^{2}\asa \ACA_{0}$ by \cite{withgusto}*{Theorem~6.9}, where the former was introduced in Section \ref{sepclo}.
To prove $\ACA_{0}\di \HEI^{2}$ or $\ACA_{0}\di \WEI^{2}$, let $f$ be RM-continuous on a separably closed set $C\subseteq [0,1]$.  Use $\TIET^{2}$ to obtain the `extension' $g$ of $f$ and use 
$\WKL$ to obtain the uniform continuity and approximation results for $g$, which all follow from \cite{simpson2}*{IV.2.3-5}.  These properties are then inherited from $g$ to $f$ on $C$, and the forward implications are done. 
Alternatively, $\ACA_{0}$ is equivalent to the Heine-Borel theorem for countable covers of separably closed sets \cite{jekke}, which readily yields $\HEI^{2}$ and $\WEI^{2}$.
 
\smallskip

For the implication $\HEI^{2}\di \ACA_{0}$, note that the antecedent for $C=[0,1]$ yields $\WKL_{0}$ by \cite{simpson2}*{IV.2.3}. 
In turn, $\WKL_{0}$ yields the \emph{strong} Tietze extension theorem for RM-continuous functions with a modulus of uniform continuity by \cite{withgusto}*{Theorem~6.14}.  
Putting together the above, we obtain $\HEI^{2}\di \TIET^{2}$ and $\ACA_{0}$ follows.  
 
\smallskip
 
For the implication $\WEI^{2}\di \ACA_{0}$, note that $\WKL_{0}$ follows from the antecedent in the same way. 
Now, $\WKL_{0}$ proves that an RM-continuous function has a modulus of uniform continuity, and the same holds for sequences of RM-continuous functions. 
This provides a modulus of uniform continuity for $f$ as in $\WEI^{2}$, and the rest of the proof is now the same as in the previous paragraph.

\smallskip

An alternative to the previous two paragraphs is provided by the second part of the proof of \cite{withgusto}*{Theorem 6.9} of $\TIET^{2}\di \ACA_{0}$.
Indeed, this proof is based on uniform continuity and $\HEI^{2}$ (more) directly applies.  The same holds for $\WEI^{2}$ by noting that polynomials are 
uniformly continuous on $[0,1]$.  Another alternative proof is based on the aforementioned result concerning the Heine-Borel theorem for countable covers of separably closed sets. 
 \end{proof}
 Thirdly, we consider the higher-order versions of $\HEI^{2}$ and $\WEI^{2}$, where the only change is that we replaced `second-order code' by `third-order function'.
\begin{princ}[$\HEI^{\omega}$]
 Let $C\subseteq[0,1]$ be separably closed.  Any $f:\R \di \R$ which is continuous on $C$ is also uniformly continuous with a modulus on $C$.
 \end{princ}
\begin{princ}[$\WEI^{\omega}$]
 Let $C\subseteq[0,1]$ be separably closed.  For $f:\R\di \R$ which is continuous on $C$, there is a sequence of polynomials $(p_{n})_{n\in \N}$ with $(\forall x\in C)(|f(x)-p_{n}(x)|<\frac{1}{2^{n}})$.
 \end{princ} 
 We now prove the following results from Figure \ref{kk}.

 \begin{thm}\label{klanty2}
 The system $\RCAo+\WKL$ proves $\HEI^{\omega}$ and $\WEI^{\omega}$.
 \end{thm}
 \begin{proof}
 We use $(\exists^{2})\vee \neg(\exists^{2})$ as in the proof of Theorem \ref{WLS}.   
In case $\neg(\exists^{2})$, all functions on $\R$ are continuous by \cite{kohlenbach2}*{\S3}.  The usual proof using $\WKL_{0}$ of the Heine and Weierstrass theorems now goes through, in light 
of the coding results in \cite{kohlenbach4}*{\S4}.  Indeed, it is established in the latter that continuous functions on $2^{\N}$ have RM-codes, and the same holds for $[0,1]$.
  
\smallskip

In case $(\exists^{2})$, we use \cite{brownphd}*{Theorem 1.10} to guarantee that any seperably closed set $C\subseteq \R$ is also closed in the sense of RM, i.e.\ given by a $\Pi_{1}^{0}$-formula in $\L_{2}$.
Note that $\exists^{2}$ can decide such formulas.  By \cite{kohlenbach4}*{\S4}, we can obtain an RM-code for $f$ on $C$, and the theorem now follows from Theorem \ref{klanty}.
\end{proof}
The following equivalence is immediate, but conceptually important.
Indeed, while the change from second-order codes to third-order functions fundamentally changes the logical strength of the Heine and Weierstrass theorems for separably closed sets, 
the higher-order versions are still equivalent to one of the Big Five. 
\begin{cor}\label{GFE}
The system $\RCAo$ proves $\HEI^{\omega}\asa \WKL\asa \WEI^{\omega}$.
\end{cor}
\begin{proof}
For $C=[0,1]$, the theorem reduces to the well-known results from \cite{simpson2}*{IV.2}.  
A RM-code defined on $[0,1]$ gives rise to a continuous third-order functional by applying $\QFAC^{1,0}$ to the totality of the RM-code. 
\end{proof}
Next, an interesting observation regarding RM is made by Montalb\'an as follows.
\begin{quote}
To study the big five phenomenon, one distinction that I think is worth making is the one between robust systems and non-robust systems. A system is robust if it is equivalent to small perturbations of itself. \cite{montahue}*{p.\ 432}
\end{quote}
As it happens, the above versions of the Tietze, Heine, and Weierstrass theorems are robust too: we can restrict to e.g.\ bounded functions and the same equivalences as in Theorem \ref{klanty} and Corollary \ref{GFE} go through.  Similarly, we could study Brown's version of the Tietze extension theorem from \cite{brownphd} equivalent to $\ACA_{0}$, and the results would be the same as for the above versions.  

\smallskip

Finally, note that Tietze's extension theorem was the first theorem for which we established the results as in Figure \ref{kk}.
Following \cite{dagsamV}, one could say that the Heine and Weierstrass theorems also `suffer from the Tietze syndrome', if the latter name were not taken already.  
\section{Ekeland's variational principle}\label{ekesect}
We establish the results pertaining to Ekeland's variational principle as sketched in Figure \ref{kk}.
In a nutshell, we show that the equivalences for fragments of this principle involving $\ACA_{0}$ and $\FIVE$ \emph{disappear} when 
we formulate this principle in higher-order arithmetic.  We shall not study the exact details of the latter formulation, only what happens \emph{if} this is done.  
\subsection{At the level of arithmetical comprehension}
We  show that the equivalence between $\ACA_{0}$ and a fragment of Ekeland's variational principle disappears when 
we formulate this principle in higher-order arithmetic. 

\smallskip

First of all, the RM-properties of Ekeland's principle are studied in \cite{ekelhaft} \emph{inside the framework of second-order arithmetic};
Ekeland's \emph{weak variational principle} from \cite{oozeivar} is called `$\FVP$' in \cite{ekelhaft}, where the following results are proved.
\begin{thm}~
\begin{enumerate}
 \renewcommand{\theenumi}{\alph{enumi}}
\item Over $\RCA_{0}$, $\ACA_{0}$ is equivalent to $\FVP$ restricted to total \emph{honestly coded} potentials on $[0,1]$ \(or $2^{\N}$\).  
\item Over $\RCA_{0}$, $\WKL_{0}$ is equivalent to $\FVP$ restricted to total \textbf{continuous} potentials on $[0,1]$ \(or $2^{\N}$\).  
\end{enumerate}
\end{thm}
Now let $\FVP^{\omega}_{\tot}$ and $\FVP^{\omega}_{\tot, \cont}$ be some \emph{higher-order} versions of $\FVP$ for total (and continuous in the second case) potentials on $[0,1]$, i.e.\ the versions of $\FVP$ from respectively items (a) and (b), but \emph{not involving codes}.  To be absolutely clear, $\FVP^{\omega}_{\tot}$ has the form $(\forall f:[0,1]\di \R)A(f)$, while $\FVP^{\omega}_{\tot, \cont}$ has the form $(\forall f\in C([0,1]))A(f)$, which is readily expressed in the language of higher-order arithmetic.  The exact formulation does not matter, as long as the aforementioned syntactical form is available.  
Moreover, since we could also use $2^{\N}$ instead of $[0,1]$, none of what follows has anything to do with the coding of real numbers.  

\smallskip

As it happens, the higher-order principles $\FVP^{\omega}_{\tot, \cont}$ and $\FVP^{\omega}_{\tot}$ cannot satisfy the same properties as in items (a) and (b) above, by the following theorem.
\begin{thm}
Assume the following proofs are given. 
\begin{enumerate}
 \renewcommand{\theenumi}{\alph{enumi}$'$}
\item The system $\ACAo$ proves $\FVP_{\tot}^{\omega}$,\label{kedna}
\item The system $\RCAo+\WKL$ proves $\FVP_{\tot,\cont}^{\omega}$.\label{foreg}
\end{enumerate}
Then $\FVP^{\omega}_{\tot}$ is already provable in $\RCAo+\WKL$.  
\end{thm}
\begin{proof}
We use the law of excluded middle $(\exists^{2})\vee \neg(\exists^{2})$.  
The proof in the former case follows thanks to item \eqref{kedna}.  In case $\neg(\exists^{2})$, all functions on $\R$ are continuous by \cite{kohlenbach4}*{\S4}.  
Hence, $\FVP^{\omega}_{\tot}$, which has the form $(\forall f:[0,1]\di \R)A(f)$, reduces to $\FVP^{\omega}_{\tot, \cont}$, which has the form $(\forall f\in C([0,1]))A(f)$.  
The latter has a proof thanks to item \eqref{foreg}, and we are done. 
\end{proof}
To be absolutely clear, we do not claim that anything is \emph{wrong} with the RM of $\FVP$ as in items (a) and (b) above.  We \emph{do} claim that no higher-order version of $\FVP$ can satisfy the same properties.  Indeed, 
items (a$'$) and (b$'$) are enough to make $\FVP_{\tot}^{\omega}$ provable in $\RCAo+\WKL$, and hence $\FVP_{\tot}^{\omega}$ \emph{cannot} imply $\ACA_{0}$.
The discrepancy between the second- and higher-order RM is (presumably) caused by the use of codes in second-order arithmetic, in particular the concept of `honest code' from item (a).  
Note that every continuous function on $\N^{\N}$ has an RM-code on $2^{\N}$ given $\WKL$ \cite{kohlenbach4}*{\S4}, i.e.\ item (b$'$) seems rather natural.  

\smallskip

The observation from this section is not an isolated incident, as we will see next.
\subsection{At the level of hyperarithmetical comprehension}\label{cook}
We  show that the equivalence between $\FIVE$ and a fragment of Ekeland's variational principle disappears when 
we formulate this principle in higher-order arithmetic. 
 
 \smallskip
 
The following results regarding $\FVP$ are from \cite{ekelhaft}.
\begin{thm}\label{hogo}
The system $\RCA_{0}$ proves the following equivalences.
\begin{enumerate}
 \renewcommand{\theenumi}{\alph{enumi}}
\item[(c)] $\FIVE$ is equivalent to $\FVP$ restricted to total potentials on $\N^{\N}$.  
\item[(d)] $\ACA_{0}$ is equivalent to $\FVP$ restricted to total \textbf{continuous} potentials on $\N^{\N}$.  
\end{enumerate}
\end{thm}
Now let $\FVP^{\omega}_{\tot}(\N^{\N})$ and $\FVP^{\omega}_{\tot, \cont}(\N^{\N})$ be some \emph{higher-order} versions of $\FVP$ for total (and continuous in the second case) potentials on $\N^{\N}$, i.e.\ the versions of $\FVP$ from respectively items (c) and (d), but \emph{not involving codes}.  To be absolutely clear, $\FVP^{\omega}_{\tot}(\N^{\N})$ has the form $(\forall f:\N^{\N}\di \N^{\N})A(f)$, while $\FVP^{\omega}_{\tot, \cont}$ has the form $(\forall f\in C(\N^{\N}))A(f)$, which is readily expressed in the language of higher-order arithmetic.

\smallskip

As it happens, the above higher-order principles cannot satisfy the same properties as in items (c) and (d) above: the following two items already imply that $\FVP^{\omega}_{\tot}(\N^{\N})$ is provable in a conservative extension of $\ACA_{0}$.  Note that $\RCAo+\ACA_{0}+(\kappa_{0}^{3})$ is conservative\footnote{The system $\RCAo+\MUC$ from \cite{kohlenbach2}*{Prop.\ 3.12} readily proves $(\kappa_{0}^{3})$, and hence the associated conservation result for the latter follows.} over $\ACA_{0}$ by \cite{kohlenbach2}*{Prop.\ 3.12}.

\begin{thm}\label{nogo2}
Assume the following proofs are given.  
\begin{enumerate}
\item[(c$'$)] The system $\Z_{2}^{\Omega}$ proves $\FVP_{\tot}^{\omega}(\N^{\N})$.
\item[(d$'$)] The system $\RCAo+\ACA_{0}$ proves $\FVP_{\tot,\cont}^{\omega}(\N^{\N})$.
\end{enumerate}
Then $\RCAo+\ACA_{0}+(\kappa_{0}^{3})$ already proves $\FVP_{\tot}(\N^{\N})$.
\end{thm}
\begin{proof}
As mentioned in \cite{dagsam}*{\S6} or \cite{samsplit}, $\RCAo$ proves $(\exists^{3})\asa [(\exists^{2})+(\kappa_{0}^{3})]$, which was first proved by Kohlenbach in a private communication.   
Now use $(\exists^{2})\vee \neg(\exists^{2})$ as in the previous proof.
\end{proof}
To be absolutely clear, we do not claim that anything is \emph{wrong} with the RM of $\FVP$ as in items (c) and (d) above.  We \emph{do} claim that no higher-order version of $\FVP$ can satisfy the same properties.  Indeed, 
items (c$'$) and (d$'$) are enough to make $\FVP_{\tot}^{\omega}(\N^{\N})$ provable in a conservative extension of $\ACA_{0}$ by Theorem \ref{nogo2}, and hence $\FVP_{\tot}^{\omega}(\N^{\N})$ \emph{cannot} imply $\FIVE$.
The discrepancy between the second- and higher-order RM is (presumably) caused by the use of codes in RM.  
Note that we use $\RCAo+\ACA_{0}$ in the theorem and item (d$'$), rather than $\ACAo$.  This does not really constitute a weakening as the former system already proves that every continuous 
functional on Baire space has an RM-code by \cite{kohlenbach4}*{\S4}.   

\smallskip

The results in Theorem \ref{nogo2} are also foundationally significant as follows: item (d) above could lead one to claim that the associated version of Ekeland's variational 
principle is not available in \emph{predicativist mathematics} (see \cite{simpson2}*{I.11.9} or \cite{fefermanmain}).  Such a claim is debatable as Theorem \ref{nogo2} provides 
a way of `pushing down' theorems to predicatively reducible mathematics.  We discuss this observation in more detail in Section \ref{FRM}.  

\smallskip

A possible criticism of Theorem \ref{nogo2} is that the axiom $(\kappa_{0}^{3})$ is somewhat ad hoc and not that natural.  
One can replace this axiom by more natural theorems as follows: the Lindel\"of lemma for Baire space together with $(\exists^{2})$ proves $\FIVE$, while the Heine-Borel theorem 
for uncountable covers of Cantor space yields $\ATR_{0}$ when combined with $(\exists^{2})$ (see \cite{dagsamIII, dagsamV} for these results).  
Thus, one readily modifies Theorem \ref{nogo2} and its proof to work with the aforementioned theorems rather than $(\kappa_{0}^{3})$. 
Note that in the case of the Lindel\"of lemma, the resulting theorem still applies to item (c) and $\FIVE$ in particular.
%
%



\section{Codes of various order and naturalness}\label{REF}
\subsection{Introduction}
The above results raise the obvious question whether third-order objects are a more natural form of representation than second-order objects.
We discuss various possible (partial) answers in this section, as follows.  

\smallskip

First of all, we discuss some relevant history of mathematics in Section \ref{histo}, namely pertaining to the modern concept of function.
We establish that the latter predates set theory and therefore should be studied in RM.  A more definitive answer shall be obtained in Section \ref{racey}.     

\smallskip

Secondly, we discuss the coding of open sets in Section \ref{kannietcopen}, motivated 
by the close connection between the coding of open sets and continuous functions in second-order RM.  
We argue that the higher-order representation of open sets is closer to the mathematical mainstream, based in part on Section \ref{histo}.

\smallskip

Thirdly, we discuss related results in Section \ref{rest}, namely pertaining to `totality versus partiality' and closely related results connected to coding from \cite{samsplit}.

\subsection{The history of functions}\label{histo}
As discussed throughout the literature, RM studies \emph{ordinary mathematics}, which is generally qualified as:
\begin{quote}
that body of mathematics that is prior to or independent of the introduction of abstract set theoretic concepts. (\cite{simpson2}*{I.1})
\end{quote}
Now, the study of arbitrary -in particular discontinuous- functions clearly predates set theory.  
Indeed, the modern concept of `arbitrary' or `general' function is generally credited to Lobachevsky (\cite{loba}) and Dirichlet (\cite{didi3}) in 1834-1837.  
A detailed history may be found in \cite{beuler}*{\S13}, where Euler is credited (much earlier) for the general definition.   
Regardless of priority, Fourier's famous work (\cite{fourierlachaud}) on Fourier series gave great impetus to the (study of the) general definition, as follows.   
\begin{quote}
\emph{This is the first of its \textup{[=Fourier series]} services which I wish to emphasize, the development and complete clarification of the concept of a function.} (\cite{vlekje}*{p.\ 116}, emphasis in original)  
\end{quote}
Moreover, discontinuous functions had already enjoyed a rich history by the time set theory came to the fore. 
Indeed, Euler's use of pulse functions, which are zero except at one point, is discussed in \cite{beuler}*{p.\ 71}, while Dirichlet famously mentions the characteristic function of $\Q$ in \cite{didi1} around 1829; Riemann studied discontinuous functions in his 1854 \emph{Habilitationsschrift}, which propelled them in the mathematical mainstream,  according to Kleiner \cite{kleine}*{p.\ 115}.   
The 1870 dissertation of Hankel, a student of Riemann, has `discontinuous functions' in its title (\cite{hankelijkheid}).  

\smallskip

The previous is not meant to settle priority disputes, but rather (only) establish that arbitrary -in particular discontinuous- functions predate set theory `by a large margin'. 
From this point of view, the study of such functions squarely belong to ordinary mathematics.  

\smallskip

\emph{In our opinion}, the aforementoned study is best undertaken in a language that has third-order objects as first-class citizens, like higher-order arithmetic.  In Section~\ref{racey}, we provide additional arguments in favour of this view, based on the fundamental connections between second- and third-order objects; such connections do not really exist  for second-order representations.  In this way, the study of $\TIET^{\omega}$ etc.\ is not an afterthought, but the very bread and butter of RM.

\subsection{Coding open sets}\label{kannietcopen}
The coding of continuous functions in second-order RM is intrinsically linked to the coding of open sets in light of \cite{simpson2}*{II.7.1}.  
The latter essentially establishes that these two concepts are (effectively) interchangeable in the base theory $\RCA_{0}$.  
Hence, the question whether third-order functions are a more natural form of coding than second-order (codes for) functions {directly} translates
to open sets.  We discuss the coding of the latter in this section.  

\smallskip

First of all, open sets are represented in second-order RM as $\Sigma_{1}^{0}$-formulas
with an additional extensionality requirement (see \cite{simpson2}*{II.5.6-7}).  Intuitively, such
formulas represent countable unions of basic open intervals.  In turn, open sets in separable spaces can 
be represented via such unions, explaining the coding.   It goes without saying that the notion of $\Sigma_{1}^{0}$-formula is a construct from mathematical logic
The aforementioned \cite{simpson2}*{II.7.1} shows that open sets in second-order RM are (equivalently) represented by (second-order codes for) continuous characteristic functions.   

\smallskip

Secondly, we have previous studied open sets represented via third-order characteristic functions in \cite{dagsamVII, dagsamVI, dagsamX}.  
Characteristic (aka indicator) functions are used to represent sets in mathematical logic in e.g.\ \cite{kruisje, hunterphd}, but \emph{also in mainstream mathematics}: most textbooks develop the Lebesgue integral using (arbitrary) characteristic functions of sets (see \cite{taomes}*{p.\ xi} for an example), while Dirichlet already considered the characteristic function of $\Q$ in \cite{didi1} in 1829. 
We also mention \emph{Thomae's function}, similar to Dirichlet's function and introduced in \cite{thomeke}*{p.\ 14} around 1875.     

\smallskip

In light of the previous, it is not an exaggeration to claim that \emph{in the case of open sets}, the third-order representation of open sets via (arbitrary) characteristic functions is `more mainstream mathematics' than the representation in second-order RM via $\Sigma_{1}^{0}$-formulas or (second-order codes for) continuous characteristic functions.  
In particular, Section \ref{histo} dictates the study of arbitrary characteristic functions in RM.  

%
%
%
%

\subsection{Related results}\label{rest}
We discuss results related to the representations of objects in second- and third-order arithmetic.  
In light of the above results on $\RCP$, an important role is played by the partial versus total distinction, as discussed in Section~\ref{parto}.
We discuss closely related RM-results connected to coding in Section~\ref{darko}; these results were first published in \cite{samsplit}. 
\subsubsection{Partial versus total objects}\label{parto}
We discuss the use of partial and total objects in mathematics motivated by $\RCP$ and Corollary \ref{fledna}.  

\smallskip

First of all, as explained nicely in \cite{zweer}*{p.\ 10}, the study of computability based on Turing's framework (\cite{tur37}) is fundamentally based on \emph{partial} computable functions for the simple reason that 
those can be listed in a computable way.  By contrast, the total computable functions are susceptible to diagonalisation \emph{by their very nature}.  
Since second-order RM makes heavy use of results in (Turing) computability theory, it is not a surprise that partial functions (like codes in $\TIE^{2}$) play an important role.  

\smallskip

Secondly, Kleene's notion of computability based on S1-S9 (\cite{kleeneS1S9}) is based on total functionals, while the system $\RCAo$ is officially a type theory in which all functionals are total. 
In general, Martin-L\"of's intuitionistic type theory (see e.g.\ \cite{loefafsteken}) is similarly based on total objects, and the same for the associated proof assistants.  One can accommodate partial functions in type theory using certain advanced techniques (see e.g.\ \cite{bove}). 

\smallskip

Thus, mathematics and computer science boast frameworks based on partial objects and frameworks based on total objects.  
The choice of framework may therefore depend on personal goals and preconceptions.  As discussed in Section~\ref{trintro}, Tietze discusses
both total and partial functions, namely as in items (A) and (B). 

\smallskip

In conclusion, the partial versus total distinction does not (directly) provide arguments for or against a particular kind of coding.  
However, as discussed next, the choice of framework may be the cause of certain observed phenomena.  

\subsubsection{Results related to coding}\label{darko}
As discussed in Section \ref{parto}, mathematics and computer science boast frameworks based on partial objects and frameworks based on total objects.  
The choice of framework may therefore depend on personal goals and preconceptions. 
In particular, one is free to chose between the development of RM in second- or higher-order arithmetic.  

\smallskip

\emph{However}, one should keep in mind that certain observed phenomena are (only) an artifact of this choice.  
As an example, we discuss \emph{splittings} and \emph{disjunctions} in RM, as studied in \cite{samsplit}. 

\smallskip

As to splittings, there are some examples in RM of theorems $A, B, C$ such that $A\asa (B\wedge C)$, i.e.\ $A$ can be \emph{split} into two independent (fairly natural) parts $B$ and $C$.  
As to disjunctions, there are (very few) examples in RM of theorems $D, E, F$ such that $D\asa (E\vee F)$, i.e.\ $D$ can be written as the \emph{disjunction} of two independent (fairly natural) parts $E$ and $F$.  
By contrast, there is a plethora of (natural) splittings and disjunctions in {higher-order} RM, as shown in \cite{samsplit} and witnessed by Figure \ref{fagure} below. 
\begin{figure}[h]
\resizebox{\linewidth}{!}{%
\begin{tabular}{ |c|c|c| } 
 \hline
 $\MUC\asa [\WKL +(\kappa_{0}^{3})+\neg(\exists^{2})]$ & $(\exists^{3})\asa [(Z^{3})+ (\exists^{2})]$ &$(\kappa_{0}^{3})\asa [(Z^{3})+\FF]$ \\ 
  $\MUC\asa [\WKL +(\kappa_{0}^{3})+\neg(S^{2})]$&$(\exists^{3})\asa [(\kappa_{0}^{3})+(\exists^{2})]$ & $[(\kappa_{0}^{3})+\WKL]\asa  [(\exists^{3})\vee \MUC]$  \\
    $\MUC\asa [\WKL +(\kappa_{0}^{3})+\neg(\exists^{3})]$& $(\exists^{3})\asa [\FF+(Z^{3})+\neg\MUC] $& $\FF\asa [(\exists^{2})\vee \MUC]$  \\
    $\MUC\asa [\FF+ \neg(\exists^{2})]$ &  $(\exists^{2})\asa [\UATR\vee\neg\HBU]$ & $\FF\asa [(\exists^{2})\vee (\kappa_{0}^{3})] $\\
    $\MUC\asa [\FF+(Z^{3})+ \neg(S^{2})]$&$(\exists^{2})\asa [\FF+\neg\MUC]$& $(Z^{3})\asa [(\exists^{3})\vee \neg(\exists^{2})]$ \\
        $\MUC\asa [\FF+(Z^{3}) +\neg(\exists^{3})]$&$ \WKL\asa [(\exists^{2})\vee \HBU] $ &$(Z^{3})\asa [(\exists^{3})\vee \neg\FF\vee \MUC]$\\
  $\T_{1}\asa [ \T_{0}\vee \Sigma_{2}^{0}\textup{\textsf{-IND}}]$      &$\WWKL\asa [(\exists^{2})\vee \WHBU]$& $\LIN\asa [\HBU \vee \neg\WKL]$ \\
        \hline
\end{tabular}
}
\caption{Summary of (some of) the results in \cite{samsplit}}\label{fagure}
\end{figure}\\
We refer to \cite{samsplit}*{\S2} for the relevant definitions not found in Appendix \ref{prelim2}. 
Some results in Figure \ref{fagure} are proved in extensions of $\RCAo$.

\smallskip

In a nutshell, splittings and disjunctions are rare in second-order RM, but rather common in higher-order arithmetic.  
As discussed at length in \cite{samsplit}*{\S6} and suggested by Figure \ref{fagure}, the splittings and disjunctions in Figure \ref{fagure} are all directly based on the axiom $(\exists^{2})$ and/or the associated `excluded middle trick'.  

\smallskip

In conclusion, as discussed at the end of Section \ref{sepclo}, the use of third-order functionals in $\TIET^{\omega}$ constitutes an implicit constructive enrichment, based on $(\exists^{2})$ and the excluded middle trick.
However, the latter are also responsible for Figure~\ref{fagure}, and more generally most results in \cite{samsplit}.

\section{Foundational implications}
We discuss the foundational implications of the above results.  
In Section \ref{FRM}, we investigate how our results affect the foundationals claims made in RM, especially pertaining to \emph{predicativist mathematics} and \emph{Hilbert's program for the foundations of mathematics}.
In Section \ref{XMX}, we speculate on \emph{why} the theorems in Figure~\ref{kk} behave as they do.  We also formulate a general criterion for judging coding.

\subsection{Lines in the sand}\label{FRM}
Certain results in RM are viewed as contributions to foundational programs by Hilbert and Russell-Weyl.  
We critically examine such claims in this section.  In a nuthshell, these programs draw a line in the sand in that 
they identify a certain threshold of logical strength above which the associated mathematics is somehow no longer meaningful or acceptable. 
Our above results are very relevant as we have shown how certain theorems can \emph{drop} in strength, even below the aforementioned thresholds as it turns out.  

\subsubsection{Hilbert's program for the foundations of mathematics}
We discuss \emph{Hilbert's program for the foundations of mathematics}, described in \cite{hilbertendlich}, an outgrowth of Hilbert's \emph{second problem} from his famous list of 23 open problems \cite{hilbertlist}.
The aim of this program was to establish the consistency of mathematics using only so-called finitistic methods.  
Tait has argued that Hilbert's finitistic mathematics is captured by the formal system $\PRA$ in \cite{tait1}.

\smallskip

Now, G\"odel's \emph{incompleteness theorems} establish that any logical system that can accommodate arithmetic, cannot even prove its own consistency \cite{godel3}.  
Thus, it is generally\footnote{By contrast, Detlefsen \cites{det1, det2} and Artemov \cite{sartemov} argue at length why G\"odel's results do not show that Hilbert's program is impossible.} believed that G\"odel's results show that Hilbert's program is impossible.  According to Simpson \cite{simpson2}*{IX.3.18}, a \emph{partial} realisation of Hilbert's program is however possible as follows: a certain theorem $T$ is called \emph{finitistically reducible} if $T$ is provable in a $\Pi_{2}^{0}$-conservative extension of $\PRA$.  The most prominent of the latter systems is the Big Five system $\WKL_{0}$.  
The intuitive idea is that while $T$ may deal with infinitary objects (and is therefore not finitistic), $T$ does not yield any (new) $\Pi_{2}^{0}$-sentences that $\PRA$ cannot prove. 
We note that Burgess criticises the above in \cite{burgess}, but we will go along with Simpson.

\smallskip

In light of the previous, the `line in the sand' is as follows: the $\Pi_{2}^{0}$-sentences provable in $\PRA$ are finitistic, and any theorem $T$ that does not prove any new $\Pi_{2}^{0}$-sentences may be called `reducible to finitism'.  
All other theorems are `not finitistically reducible'.  
Our results pose a problem as follows: the Tietze, Heine, and Weierstrass theorems for separably closed are \emph{not} finitistically reducible when formulated with second-order codes (Theorem \ref{klanty}), while these theorems drop to finitistically 
reducible when formulated in third-order arithmetic (Theorems~\ref{WLS} and~\ref{klanty2}).  
In our opinion, the foundational status of a theorem, i.e.\ finitistically reducible or not, should not depend on technical details like coding.    

\smallskip

%

\subsubsection{Predicativist mathematics}
We discuss the development of \emph{predicativst mathematics} due to Russell,  Weyl, and Ferferman \cites{weyldas, fefermanga, rukker}.
In a nutshell, motivated by contradictions in naive set theory, predicativist mathematics also draws a line in the sand, namely at the level\footnote{As in the previous section, some people disagree and claim that predicativist mathematics can extend beyond $\ATR_{0}$ \cite{weefgetouwtrek}.} of $\ATR_{0}$.  Our results from Section \ref{cook} therefore 
yield a `drop' similar to the previous section, namely from the level of $\FIVE$ (above $\ATR_{0}$) to level of $\ACA_{0}$ (below $\ATR_{0}$). 

\smallskip

First of all, \emph{Russell's paradox} shows that naive set theory is inconsistent, the problematic entity being the `set of all sets'.  
The root cause of this paradox, according to Russell, is \emph{implicit definition}: the `set of all sets' is defined as the collection 
of all sets \emph{including itself} \cite{rukker}.  Thus, to define the `set of all sets', one implicitly assumes that it exists already, a kind of \emph{vicious circle}.  

\smallskip

According to Russell, one should in general avoid such implicit definitions and vicious circles, which are also called \emph{impredicative}, or \emph{non-predicative} in \cite{rukker}, leading to the term \emph{predicativist} or \emph{predicative} for 
that mathematics that makes no use of impredicative notions.
Weyl's \emph{Das Kontinuum} is a milestone in the development of mathematics on predicative grounds \cite{weyldas}, while the main protagonist of the modern development is Feferman \cite{fefermanlight}.  

\smallskip

The upper limit of predicativist mathematics was identified independently by Feferman and Sch\"utte as being the ordinal $\Gamma_{0}$, which is also the proof-theoretic ordinal of the Big Five system $\ATR_{0}$ \cites{simpson2, fefermanmain}.  Similar to the previous section, systems like $\ACAo$ may involve impredicative notions, but since its ordinal is (well) below $\Gamma_{0}$, the former system is called \emph{predicatively reducible}.

\smallskip

Having found the line drawn in the sand by predicativist mathematics, we note that our results pose a problem: 
the Ekeland variational principle for $\N^{\N}$ is \emph{not} predciatively reducible when formulated with second-order codes (Theorem \ref{hogo}), 
while this theorems drop to predicatively reducible when formulated in third-order arithmetic (Theorem~\ref{nogo2}).  
We repeat that, in our opinion, the foundational status of a theorem, i.e.\ predicatively reducible or not, should not depend on technical details like coding.

\subsection{Cause and effect}\label{XMX}
In this section, we speculate on \emph{why} the theorems in Figure~\ref{kk} behave as they do.  
We first discuss Kohlenbach's result on coding in Section \ref{koko}, which yields 
a satisfactory answer to the aforementioned question `by contrast'.  
As suggested above, this understanding should ideally lead to a distinction between `good' and `bad' codes, as e.g.\ the coding of real numbers does not seem to lead to results as in Figure \ref{kk}.  
We discuss the `good' versus `bad' distinction for codes in Section \ref{racey}.
\subsubsection{Coding comes to a head}\label{koko}
We sketch the `coding' results from \cite{kohlenbach4}*{\S4} and identify the cause of the behaviour of the theorems in Figure \ref{kk} in contrast to Kohlenbach's results.

\smallskip

Kohlenbach proves in \cite{kohlenbach4}*{\S4} that for continuous type two functionals, the existence of a RM-code is equivalent to the existence of a continuous modulus of continuity.  
Thus, RM-codes constitute a constructive enrichment as follows: when interpreted in higher-order arithmetic, a second-order theorem about RM-continuous functions only seems to apply to higher-order functionals \emph{that also come with an RM-code \(or the aforementioned modulus\)}.  In particular, it is not clear whether such a theorem applies to \emph{all} continuous functionals: `out of the box' it only seems to apply to the sub-class `continuous functionals with a continuous modulus'.  
To remove (part of) this uncertainty, note that $\WKL$ suffices to show that continuous functionals have an RM-code on $2^{\N}$ \cite{kohlenbach4}*{\S4}, i.e.\ the RM of $\WKL$ does not change if we replace any leading universal quantifier over RM-codes by a quantifier over higher-order functionals that are continuous in the relevant theorems.    

\smallskip

The observation in the previous paragraph gives rise to the concept of \emph{coding overhead} of a given theorem $T$, which is the minimal axioms(s) needed to prove (over $\RCAo$) that $T$ formulated \emph{with second-order codes} implies $T$ formulated \emph{in third-order arithmetic}, whenever the former seems less general than the latter.  However, coding overhead is not all there is, much to our own surprise, as follows.  

\smallskip

The problem with the theorems in Figure \ref{kk} turns out to be \emph{the opposite} of the direction indicated by the concept of coding overhead: our results for e.g.\ the Tietze extension 
theorem imply that, working again in Kohlenbach's base theory, the class of `second-order codes defined on a separably closed set $C$' is much \emph{bigger} than `third-order functions continuous on $C$'. 
In this way, statements like $\RCP$ connecting the aforementioned classes are at the level of $\ACA_{0}$, while the associated version of Tietze's extension theorem $\TIET^{\omega}$ is provable in $\RCAo$.
In particular, Theorem \ref{winkel} suggests we also define the notion of \emph{coding underhead} of a theorem $T$, which is the minimal axioms(s) needed to prove (over $\RCAo$) that $T$ formulated \emph{in third-order arithmetic} implies $T$ formulated \emph{with second-order codes}, and this whenever the former seems less general than the latter, like in the case of separably closed sets. 

\smallskip

In conclusion, past results on coding have focused on the constructive enrichment provided by codes, as this causes the class of represented objects to be (potentially) smaller than the class of actual objects.
By contrast, the second-order Tietze, Weierstrass, and Heine theorems in Figure \ref{kk} pertain to a class of codes that is \emph{larger} than the class of third-order objects (working in $\RCAo$).  
This `opposite' difference in size is the cause of the behaviour observed in Figure \ref{kk}.  

%
%

\subsubsection{Beyond good and bad coding}\label{racey}
Having identified the cause of the behaviour of the theorems in Figure \ref{kk}, we now expand our point of view and attempt 
to provide a criterion by which coding can be judged as `good' or `bad', in light of the Main Question of RM.  We reiterate that we do not 
judge the coding practise of RM as good or bad.  We only wish to establish a criterion that avoids the `bad' behaviour of coding as summarised 
in Figure \ref{kk}, assuming one judges the latter as such.

\smallskip

Firstly, we consider the following -in our opinion- uncontroversial statement.
\begin{quote}
Codes for a certain class of objects are meant to capture/represent/reflect, as well as possible, the original class of objects.  
\end{quote}
Simpson's quotes on `theorems as they stand' from Section \ref{intro} support this statement, while the following quote expresses the same idea.
\begin{quote}  
There is a key criterion for choice of coding that is implicitly used: that $\RCA_{0}$
should prove as much as possible. E.g., if real numbers are Cauchy sequences of
rationals, we can't prove in $\RCA_{0}$ that every real number is the limit of a sequence
of rationals with arbitrarily fast convergence. \cite{fried5}*{p.\ 134}
\end{quote}
Indeed, this criterion is used to choose the particular coding of real numbers used in RM, namely fast-converging Cauchy sequences as in \cite{simpson2}*{II.4}.  
By contrast, other possible representations of real numbers (like Dedekind cuts) are rejected because $\RCA_{0}$ cannot prove basic properties about the representations.  

\smallskip

The problem with the previous is that one does not \emph{fully} take the role 
of the base theory into account.  Indeed, in our opinion, the following addition is crucial.  
\begin{quote}
Codes for a certain class of objects are meant to capture/represent/reflect, as well as possible, the original class of objects \emph{also in non-classical extensions of the base theory}.
\end{quote}
Our rationale is that the base theories $\RCA_{0}$ and $\RCAo$ are weak enough to be consistent with non-classical axioms like e.g.\ the axiom \textsf{CT} which expresses that every $\N\di \N$-function is computable (in the sense of Turing). 
The axiom \textsf{CT} is called \emph{Church's thesis} in constructive mathematics \cite{beeson1}.   Also found in the latter is `Brouwer's principle' \textsf{BP}, which expresses that all functions on $\N^{\N}$ are continuous; \textsf{BP} 
yields a conservative extension of $\RCAo$, while the associated \emph{intuitionistic fan functional} $\MUC$ from \cite{kohlenbach2}*{\S3} yields a conservative extension of $\WKL_{0}$.
Ishihara connects \textsf{BP} to $\neg(\exists^{2})$, where the latter is formulated with $\Pi_{1}^{0}$-formulas in \cite{goodforishi}.

\smallskip

As it happens, both $\textsf{CT}$ and $\textsf{BP}$ imply the \emph{higher-order} versions of the Tietze theorem from Figure \ref{kk}. 
Similarly, the higher-order Heine and Weierstrass theorems from Figure \ref{kk} follow from $\MUC$ or $\WKL+\textsf{BP}$.
This behaviour is obviously \emph{not} reflected in the associated \emph{second-order} versions form Figure \ref{kk}, as e.g.\ \textsf{CT} and $\ACA_{0}$ are inconsistent.  
It should be noted that non-classical axioms like $\neg \WKL_{0}$ have been studied in classical RM in \cite{boulanger}. 

\smallskip

Another way of looking things based on Corollary \ref{fledna} is as follows.  One one hand, $\TIET^{\omega}$ follows from $\CT$, while on the other hand $\RCP$ is inconsistent with $\CT$. 
This suggests that $\TIET^{\omega}$ is excluding the genuinely `difficult' instances of $\TIET^{2}$, which contributes to the latter being a stronger principle.  

\smallskip

\noindent The main point of the previous observations is now two-fold, as follows.

\smallskip

On one hand, there are (very) different models of weak systems, which is reflected in the syntax by $(\exists^{2})\vee \neg(\exists^{2})$ being an axiom of $\RCAo$.
Depending on which case we are working with, the class of (continuous) functions looks very different: in case $\neg(\exists)$, all functions on $\R$ are continuous, similar to Brouwer's principle $\textsf{BP}$, while
in case $(\exists^{2})$, discontinuous functions are directly available.   
\emph{By contrast}, whether we work in the case $(\exists^{2})$ or $\neg(\exists^{2})$ has less 
 of an effect (or even none) on the class of codes for (continuous) functions: in either case there is a code for a discontinuous function, namely as as a sequence of functions that converges to $\exists^{2}$ in $\RCAo$.

\smallskip

On the other hand, there is a deep connection between second- and third-order objects: $\CT$ is a statement about $\N\di \N$-functions, but nonetheless 
causes all $\N^{\N}\di \N$-functions to be continuous.  Similarly, $(\exists^{2})$ is equivalent to the existence of a discontinuous function on $\R$, and the former 
allows us to solve the Halting problem, which contradicts $\CT$.  \emph{By contrast}, the connection between second-order objects and \emph{representations} of third-order objects is 
much weaker.  Indeed, in either case of $\CT \vee \neg\CT$ there is a code for a discontinuous function.  

\smallskip

In light of the previous, we can say that (i) the class of codes for (continuous) functions is `more static' than the class of (continuous) functions itself and (ii) the class of codes for (continuous) functions is `more disconnected' from the class of second-order objects than the class of (continuous) functions itself.  

\smallskip

In conclusion, we may call a coding `good' (meaning it avoids results as in Figure~\ref{kk}) if it satisfies the previous centred statement, i.e.\ if the correspondence between objects and representations 
is also valid in non-classical extensions of the base theory.  

\subsection{White horses and horses that are white}
We discuss possible future research topics on coding via existing examples. 

\smallskip

Firstly, while it is hard to argue with the statement 
\begin{center}
\emph{a countable set is a set that is countable}, 
\end{center}
it should be noted that a countable set in $\L_{2}$ is actually a sequence by \cite{simpson2}*{V.4.2}.  The latter definition can be said to constitute a constructive enrichment compared to the usual definition based on injections or bijections.    
Of course, the usual definition of countable set (for subsets of Baire space), 
is fundamentally third-order and cannot even be expressed in $\L_{2}$.  Nonetheless, it is an interesting question what happens to e.g.\
theorems from RM if we work in higher-order RM and use the higher-order definition of countable set.  This question becomes all the more pertinent in light of the observation that textbooks tend to treat countable sets as follows:  to \emph{prove that} a given set is countable, it is generally only shown that there is an injection to $\N$, while to \emph{prove something about} countable sets, one additionally assumes the latter are given by a sequence. 

\smallskip

Secondly, a partial answer to the previous question has been published in \cite{samrecount, samflo2}.  In particular, an interesting `coding' result is obtained, as follows.
Now, the following $\L_{2}$-sentence is equivalent to $\WKL_{0}$ by \cite{simpson2}*{IV.4.5}.
\begin{princ}[$\textsf{ORD}^{2}$]
Any formally real countable field is orderable.
\end{princ}
Let $\textsf{ORD}^{3}$ be the previous principle formulated in higher-order arithmetic \emph{for \textbf{\textup{countable}} subsets of Baire space}, where \textbf{countable} is interpreted via the usual definition involving injections to $\N$, as can be found in Kunen's textbook \cite{kunen}.
Note that this is the only modification to $\ORD^{2}$.
Then $\ORD^{3}$ implies the Heine-Borel theorem for \emph{uncountable covers}, which is not provable in $\SIXK$ for any $k$.

\smallskip

Thus, the definition of `countable' used in RM also involves constructive enrichments in that the usual definition of `countable' gives rise to a wholly different beast.  
These results are not as refined or fundamental as the ones in Figure \ref{kk}.

\smallskip

Thirdly, the \emph{closed graph theorem} is provable in an extension $\RCA_{0}^{+}$ of $\RCA_{0}$ \cite{browsi}*{\S5}.
The former theorem expresses that a linear operator with a graph \emph{given by a separably closed set} must be continuous.  
Since $\ACA_{0}$ proves $\RCA_{0}^{+}$, the usual `excluded middle trick' yields a higher-order version of the closed graph theorem \emph{inside $\RCAo$}.

\smallskip

Fourth, Dag Normann and the author have obtained a lot of results about the RM of countable sets which can be found in the preprints \cites{dagsamX, samNEO2, samcount}.  
Perhaps the most interesting example is as follows.  The following two principles are equivalent over $\RCAo$, where `countable set' is defined as above, i.e.\ based on injections to $\N$.
\begin{princ}[$\cocode_{0}$]
For any non-empty countable set $A\subseteq [0,1]$, there is a sequence $(x_{n})_{n\in \N}$ in $A$ such that $(\forall x\in \R)(x\in A\asa (\exists n\in \N)(x_{n}=_{\R}x))$.
\end{princ}
\begin{princ}[$\BW_{0}^{C}$]\label{comp}
For any countable $A\subset 2^{\N}$ and $F:2^{\N}\di 2^{\N}$, the supremum $\sup_{f\in A}F(f)$ exists. 
\end{princ}
Many similar equivalences can be found in \cite{samcount}.
Moreover, combining $\BW_{0}^{C}$ with the Suslin functional, i.e.\ higher-order $\FIVE$, one obtains $\SIX$.  
The same holds at the computational level, as shown in \cite{dagsamX}.  The system $\SIX$ has been called
the current `upper limit' for RM, previously only reachable via topology (see \cites{mummymf, mummyphd, mummy}).  
Moreover, according to Rathjen \cite{rathjenICM}*{\S3}, the strength of $\SIX$ \emph{dwarfs} that of $\FIVE$.  

\begin{ack}\rm
Our research was supported by the John Templeton Foundation via the grant \emph{a new dawn of intuitionism} with ID 60842.
We express our gratitude towards this institution. 
We thank Anil Nerode, Paul Shafer, and Keita Yokoyama for their valuable advice.  
We also thank the anonymous referee for the many helpful suggestions, especially for suggesting the need for Section~\ref{REF}.  
Opinions expressed in this paper do not necessarily reflect those of the John Templeton Foundation.    
\end{ack}

\appendix
\section{Reverse Mathematics}\label{janarn}

We introduce \emph{Reverse Mathematics} in Section \ref{prelim1}, as well as its generalisation to \emph{higher-order arithmetic}, and the associated base theory $\RCAo$ of Kohlenbach's \emph{higher-order} RM.  
We introduce some essential axioms in Section~\ref{prelim2}.

\subsection{Reverse Mathematics}\label{prelim1}
Reverse Mathematics is a program in the foundations of mathematics initiated around 1975 by Friedman \cites{fried,fried2} and developed extensively by Simpson \cite{simpson2}.  
The aim of RM is to identify the minimal axioms needed to prove theorems of ordinary, i.e.\ non-set theoretical, mathematics.   In almost all cases, these minimal axioms are also \emph{equivalent} to the theorem at hand (over a weak logical system).  The derivation of the minimal axioms from the theorem is the `reverse' way of doing mathematics, lending the subject its name. 

\smallskip

We refer to \cite{stillebron} for an introduction to RM and to \cite{simpson2, simpson1} for an overview of RM.  We expect basic familiarity with RM, but do sketch some aspects of Kohlenbach's \emph{higher-order} RM \cite{kohlenbach2}`'essential to this paper, including the base theory $\RCAo$ (Definition \ref{kase}).  

\smallskip

First of all, in contrast to `classical' RM based on \emph{second-order arithmetic} $\Z_{2}$, higher-order RM uses $\L_{\omega}$, the richer language of \emph{higher-order arithmetic}.  
Indeed, while $\Z_{2}$ is restricted to natural numbers and sets of natural numbers, higher-order arithmetic can accommodate sets of sets of natural numbers, sets of sets of sets of natural numbers, et cetera.  
To formalise this idea, we introduce the collection of \emph{all finite types} $\mathbf{T}$, defined by the two clauses:
\begin{center}
(i) $0\in \mathbf{T}$   and   (ii)  if $\sigma, \tau\in \mathbf{T}$ then $( \sigma \di \tau) \in \mathbf{T}$,
\end{center}
where $0$ is the type of natural numbers, and $\sigma\di \tau$ is the type of mappings from objects of type $\sigma$ to objects of type $\tau$.
In this way, $1\equiv 0\di 0$ is the type of functions from numbers to numbers, and where  $n+1\equiv n\di 0$.  Viewing sets as given by characteristic functions, we note that $\Z_{2}$ only includes objects of type $0$ and $1$.    

\smallskip

Secondly, the language $\L_{\omega}$ includes variables $x^{\rho}, y^{\rho}, z^{\rho},\dots$ of any finite type $\rho\in \mathbf{T}$.  Types may be omitted when they can be inferred from context.  
The constants of $\L_{\omega}$ include the type $0$ objects $0, 1$ and $ <_{0}, +_{0}, \times_{0},=_{0}$  which are intended to have their usual meaning as operations on $\N$.
Equality at higher types is defined in terms of `$=_{0}$' as follows: for any objects $x^{\tau}, y^{\tau}$, we have
\be\label{aparth}
[x=_{\tau}y] \equiv (\forall z_{1}^{\tau_{1}}\dots z_{k}^{\tau_{k}})[xz_{1}\dots z_{k}=_{0}yz_{1}\dots z_{k}],
\ee
if the type $\tau$ is composed as $\tau\equiv(\tau_{1}\di \dots\di \tau_{k}\di 0)$.  
Furthermore, $\L_{\omega}$ also includes the \emph{recursor constant} $\mathbf{R}_{\sigma}$ for any $\sigma\in \mathbf{T}$, which allows for iteration on type $\sigma$-objects as in the special case \eqref{special}.  Formulas and terms are defined as usual.  
One obtains the sub-language $\L_{n+2}$ by restricting the above type formation rule to produce only type $n+1$ objects (and related types of similar complexity).        
\bdefi\label{kase} 
The base theory $\RCAo$ consists of the following axioms.
\begin{enumerate}
 \renewcommand{\theenumi}{\alph{enumi}}
\item  Basic axioms expressing that $0, 1, <_{0}, +_{0}, \times_{0}$ form an ordered semi-ring with equality $=_{0}$.
\item Basic axioms defining the well-known $\Pi$ and $\Sigma$ combinators (aka $K$ and $S$ in \cite{avi2}), which allow for the definition of \emph{$\lambda$-abstraction}. 
\item The defining axiom of the recursor constant $\mathbf{R}_{0}$: For $m^{0}$ and $f^{1}$: 
\be\label{special}
\mathbf{R}_{0}(f, m, 0):= m \textup{ and } \mathbf{R}_{0}(f, m, n+1):= f(n, \mathbf{R}_{0}(f, m, n)).
\ee
\item The \emph{axiom of extensionality}: for all $\rho, \tau\in \mathbf{T}$, we have:
\be\label{EXT}\tag{$\textsf{\textup{E}}_{\rho, \tau}$}  
(\forall  x^{\rho},y^{\rho}, \varphi^{\rho\di \tau}) \big[x=_{\rho} y \di \varphi(x)=_{\tau}\varphi(y)   \big].
\ee 
\item The induction axiom for quantifier-free\footnote{To be absolutely clear, variables (of any finite type) are allowed in quantifier-free formulas of the language $\L_{\omega}$: only quantifiers are banned.} formulas of $\L_{\omega}$.
\item $\QFAC^{1,0}$: The quantifier-free Axiom of Choice as in Definition \ref{QFAC}.
\end{enumerate}
\edefi
\bdefi\label{QFAC} The axiom $\QFAC$ consists of the following for all $\sigma, \tau \in \textbf{T}$:
\be\tag{$\QFAC^{\sigma,\tau}$}
(\forall x^{\sigma})(\exists y^{\tau})A(x, y)\di (\exists Y^{\sigma\di \tau})(\forall x^{\sigma})A(x, Y(x)),
\ee
for any quantifier-free formula $A$ in the language of $\L_{\omega}$.
\edefi
%
%
As discussed in \cite{kohlenbach2}*{\S2}, $\RCAo$ and $\RCA_{0}$ prove the same sentences `up to language' as the latter is set-based and the former function-based.  Recursion as in \eqref{special} is called \emph{primitive recursion}; the class of functionals obtained from $\mathbf{R}_{\rho}$ for all $\rho \in \mathbf{T}$ is called \emph{G\"odel's system $T$} of all (higher-order) primitive recursive functionals.  

\smallskip

We use the usual notations for natural, rational, and real numbers, and the associated functions, as introduced in \cite{kohlenbach2}*{p.\ 288-289}.  
\begin{defi}[Real numbers and related notions in $\RCAo$]\label{keepintireal}\rm~
\begin{enumerate}
 \renewcommand{\theenumi}{\alph{enumi}}
\item Natural numbers correspond to type zero objects, and we use `$n^{0}$' and `$n\in \N$' interchangeably.  Rational numbers are defined as signed quotients of natural numbers, and `$q\in \Q$' and `$<_{\Q}$' have their usual meaning.    
\item Real numbers are coded by fast-converging Cauchy sequences $q_{(\cdot)}:\N\di \Q$, i.e.\  such that $(\forall n^{0}, i^{0})(|q_{n}-q_{n+i}|<_{\Q} \frac{1}{2^{n}})$.  
We use Kohlenbach's `hat function' from \cite{kohlenbach2}*{p.\ 289} to guarantee that every $q^{1}$ defines a real number.  
\item We write `$x\in \R$' to express that $x^{1}:=(q^{1}_{(\cdot)})$ represents a real as in the previous item and write $[x](k):=q_{k}$ for the $k$-th approximation of $x$.    
\item Two reals $x, y$ represented by $q_{(\cdot)}$ and $r_{(\cdot)}$ are \emph{equal}, denoted $x=_{\R}y$, if $(\forall n^{0})(|q_{n}-r_{n}|\leq {2^{-n+1}})$. Inequality `$<_{\R}$' is defined similarly.  
We sometimes omit the subscript `$\R$' if it is clear from context.           
\item Functions $F:\R\di \R$ are represented by $\Phi^{1\di 1}$ mapping equal reals to equal reals, i.e.\ satisfying $(\forall x , y\in \R)(x=_{\R}y\di \Phi(x)=_{\R}\Phi(y))$.\label{EXTEN}
\item The relation `$x\leq_{\tau}y$' is defined as in \eqref{aparth} but with `$\leq_{0}$' instead of `$=_{0}$'.  Binary sequences are denoted `$f^{1}, g^{1}\leq_{1}1$', but also `$f,g\in C$' or `$f, g\in 2^{\N}$'.  Elements of Baire space are given by $f^{1}, g^{1}$, but also denoted `$f, g\in \N^{\N}$'.
\item For a binary sequence $f^{1}$, the associated real in $[0,1]$ is $\r(f):=\sum_{n=0}^{\infty}\frac{f(n)}{2^{n+1}}$.\label{detrippe}
\item Sets of type $\rho$ objects $X^{\rho\di 0}, Y^{\rho\di 0}, \dots$ are given by their characteristic functions $F^{\rho\di 0}_{X}\leq_{\rho\di 0}1$, i.e.\ we write `$x\in X$' for $ F_{X}(x)=_{0}1$. \label{koer} 
\end{enumerate}
\end{defi}
\noindent
Next, we mention the highly useful $\ECF$-interpretation. 
\begin{rem}[The $\ECF$-interpretation]\label{ECF}\rm
The (rather) technical definition of $\ECF$ may be found in \cite{troelstra1}*{p.\ 138, \S2.6}.
Intuitively, the $\ECF$-interpretation $[A]_{\ECF}$ of a formula $A\in \L_{\omega}$ is just $A$ with all variables 
of type two and higher replaced by countable representations of continuous functionals.  Such representations are also (equivalently) called `associates' or `RM-codes' \cite{kohlenbach4}*{\S4}. 
The $\ECF$-interpretation connects $\RCAo$ and $\RCA_{0}$ \cite{kohlenbach2}*{Prop.\ 3.1} in that if $\RCAo$ proves $A$, then $\RCA_{0}$ proves $[A]_{\ECF}$, again `up to language', as $\RCA_{0}$ is 
formulated using sets, and $[A]_{\ECF}$ is formulated using types, namely only using type zero and one objects.  
\end{rem}
In light of the widespread use of codes in RM and the common practise of identifying codes with the objects being coded, it is no exaggeration to refer to $\ECF$ as the \emph{canonical} embedding of higher-order into second-order RM.  
%
%
%

\subsection{Some axioms of higher-order RM}\label{prelim2}
We introduce some functionals which constitute the counterparts of some of the Big Five systems, in higher-order RM.
We use the formulation from \cite{kohlenbach2, dagsamIII}.  
First of all, $\ACA_{0}$ is readily derived from:
\begin{align}\label{mu}\tag{$\mu^{2}$}
(\exists \mu^{2})(\forall f^{1})\big[ (\exists n)(f(n)=0) \di [(f(\mu(f))=0)&\wedge (\forall i<\mu(f))f(i)\ne 0 ]\\
& \wedge [ (\forall n)(f(n)\ne0)\di   \mu(f)=0]    \big], \notag
\end{align}
and $\ACA_{0}^{\omega}\equiv\RCAo+(\mu^{2})$ proves the same sentences as $\ACA_{0}$ by \cite{hunterphd}*{Theorem~2.5}.   The (unique) functional $\mu^{2}$ in $(\mu^{2})$ is also called \emph{Feferman's $\mu$} \cite{avi2}, 
and is clearly \emph{discontinuous} at $f=_{1}11\dots$; in fact, $(\mu^{2})$ is equivalent to the existence of $F:\R\di\R$ such that $F(x)=1$ if $x>_{\R}0$, and $0$ otherwise \cite{kohlenbach2}*{\S3}, and to 
\be\label{muk}\tag{$\exists^{2}$}
(\exists \varphi^{2}\leq_{2}1)(\forall f^{1})\big[(\exists n)(f(n)=0) \asa \varphi(f)=0    \big]. 
\ee
\noindent
Secondly, $\FIVE$ is readily derived from the following sentence:
\be\tag{$\SS^{2}$}
(\exists\SS^{2}\leq_{2}1)(\forall f^{1})\big[  (\exists g^{1})(\forall n^{0})(f(\overline{g}n)=0)\asa \SS(f)=0  \big], 
\ee
and $\FIVE^{\omega}\equiv \RCAo+(\SS^{2})$ proves the same $\Pi_{3}^{1}$-sentences as $\FIVE$ by \cite{yamayamaharehare}*{Theorem 2.2}.   The (unique) functional $\SS^{2}$ in $(\SS^{2})$ is also called \emph{the Suslin functional} \cite{kohlenbach2}.
By definition, the Suslin functional $\SS^{2}$ can decide whether a $\Sigma_{1}^{1}$-formula as in the left-hand side of $(\SS^{2})$ is true or false.   We similarly define the functional $\SS_{k}^{2}$ which decides the truth or falsity of $\Sigma_{k}^{1}$-formulas; we also define 
the system $\SIXK$ as $\RCAo+(\SS_{k}^{2})$, where  $(\SS_{k}^{2})$ expresses that $\SS_{k}^{2}$ exists.  Note that we allow formulas with \emph{function} parameters, but \textbf{not} \emph{functionals} here.
In fact, Gandy's \emph{Superjump} \cite{supergandy} constitutes a way of extending $\FIVE^{\omega}$ to parameters of type two.  We identify the functionals $\exists^{2}$ and $\SS_{0}^{2}$ and the systems $\ACAo$ and $\SIXK$ for $k=0$.

\smallskip

\noindent
Thirdly, full second-order arithmetic $\Z_{2}$ is readily derived from $\cup_{k}\SIXK$, or from:
\be\tag{$\exists^{3}$}
(\exists E^{3}\leq_{3}1)(\forall Y^{2})\big[  (\exists f^{1})Y(f)=0\asa E(Y)=0  \big], 
\ee
and we therefore define $\Z_{2}^{\Omega}\equiv \RCAo+(\exists^{3})$ and $\Z_{2}^\omega\equiv \cup_{k}\SIXK$, which are conservative over $\Z_{2}$ by \cite{hunterphd}*{Cor.\ 2.6}. 
Despite this close connection, $\Z_{2}^{\omega}$ and $\Z_{2}^{\Omega}$ can behave quite differently, as discussed in e.g.\ \cite{dagsamIII}*{\S2.2}.   The functional from $(\exists^{3})$ is also called `$\exists^{3}$', and we use the same convention for other functionals. 
 Note that $(\exists^{3})\asa [(\exists^{2})+(\kappa_{0}^{3})]$ as shown in \cite{samsplit,dagsam}, where the latter is comprehension on $2^{\N}$: 
\be\tag{$\kappa_{0}^{3}$}
(\exists \kappa_{0}^{3}\leq_{3}1)(\forall Y^{2})\big[\kappa_{0}(Y)=0\asa (\exists f\in C)Y(f)=0  \big].
\ee  
Other `splittings' are studied in \cite{samsplit}, including $(\kappa_{0}^{3})$.

\bye